\documentclass[11pt, reqno]{amsart}

\usepackage[margin=1.3in]{geometry}
\usepackage{amsmath, amssymb, amsthm}
\usepackage{amsfonts, amscd, bm, cancel}
\usepackage{amsrefs}
\usepackage{xcolor}
\usepackage[colorlinks=true, urlcolor=blue, linkcolor=blue, citecolor=magenta, pdfstartview=FitH]{hyperref}

\newtheorem{theorem}{Theorem}[section]
\newtheorem{lemma}[theorem]{Lemma}

\newtheorem{corollary}[theorem]{Corollary}

\numberwithin{equation}{section}
\allowdisplaybreaks
\begin{document}
\title[Fenchel-Willmore inequality for submanifolds]{Fenchel-Willmore inequality for submanifolds in manifolds with non-negative $k$-Ricci curvature}

\author[M. Ji]{Meng Ji}
\address{School of Mathematics and Applied Statistics,
University of Wollongong\\
NSW 2522, Australia}
\email{\href{mailto:mengj@uow. edu.au}{mengj@uow.edu.au}}

\author[K.-K. Kwong]{Kwok-Kun Kwong}
\address{School of Mathematics and Applied Statistics,
University of Wollongong\\
NSW 2522, Australia}
\email{\href{mailto:kwongk@uow.edu.au}{kwongk@uow.edu.au}}


\begin{abstract}
We establish a sharp Fenchel-Willmore inequality for closed submanifolds of arbitrary dimension and codimension immersed in a complete Riemannian manifold with non-negative intermediate Ricci curvature and Euclidean volume growth. In the hypersurface case, this reduces to non-negative Ricci curvature. We also characterize the equality case. This generalizes the recent work of Agostiniani, Fogagnolo, and Mazzieri \cite{Agostiniani-Fogagnolo-Mazzieri}, as well as classical results by Chen, Fenchel, Willmore, and others.
\end{abstract}
\maketitle

\section{Introduction}\label{sec intro}
The celebrated Willmore inequality \cite{Willmore68} states that for a bounded domain $\Omega$ of $\mathbb{R}^3$ with smooth boundary, we have
$$
\int_{\partial \Omega} \sigma^2 \ge 4\pi,
$$
where $\sigma=\frac{1}{2}\text{tr }\mathrm{II}$ is the normalized mean curvature of $\partial \Omega$ and $\mathrm{II}$ denotes the second fundamental form.
The inequality is sharp and the equality holds if and only if $\Omega$ is a ball.

Another fundamental geometric inequality, Fenchel inequality \cite{Fenchel29}, asserts that for any closed, embedded curve $\Sigma$ in $\mathbb R^3$, the integral of the length of its curvature vector $\kappa$ satisfies
$$\int_{\Sigma}|\kappa|\ge 2\pi, $$
with equality only for planar convex curves.

These inequalities have motivated extensive research. Notably, in a relatively recent influential work, Agostiniani, Fogagnolo, and Mazzieri \cite{Agostiniani-Fogagnolo-Mazzieri} extended the classical Willmore inequality to the boundary of a smooth, bounded open set $\Omega$ in a complete $n+1$ dimensional manifold $M$ ($n\ge 2$) with nonnegative Ricci curvature and Euclidean volume growth. They established the following inequality:
\begin{equation*}
\int_{\partial \Omega}\left|\sigma\right|^{n} \ge \theta\left|\mathbb{S}^{n}\right|,
\end{equation*}
where $\sigma=\frac{1}{n}\mathrm{II}$ is the normalized mean curvature vector and $\theta$ is the asymptotic volume ratio defined by (which is well-defined and independent of $p$ by the Bishop-Gromov volume comparison theorem)
$$\theta=\lim_{r \rightarrow \infty} \frac{ |B(p, r)|}{|\mathbb{B}^{n+1}|r^{n+1}}. $$
Moreover, equality holds if and only if $M \setminus \Omega$ is isometric to $\left(\left[r_0, \infty\right) \times \partial \Omega, d r^2+\left(\frac{r}{r_0}\right)^2 g_{\partial \Omega}\right)$ with $r_0=\left(\frac{|\partial \Omega|}{\theta\left|\mathbb{S}^n\right|}\right)^{\frac{1}{n}}$.
In particular, $\partial \Omega$ is a connected umbilical hypersurface with constant mean curvature.
Here, $\mathbb{B}^{n+1} \subset \mathbb{R}^{n+1}$ denotes the standard unit ball, and $\mathbb{S}^{n} = \partial \mathbb{B}^{n+1}$ is the unit round sphere.

The result in \cite{Agostiniani-Fogagnolo-Mazzieri}, derived via a level-set method based on potential theory, has drawn significant interest due to its geometric implications. Wang \cite{Wang23} subsequently provided a more concise argument based on classical comparison techniques. Both works are confined to closed hypersurfaces that form the boundary of open bounded subsets. Generalizing the level-set approach to an immersed $\Sigma$ or to higher codimension appears to be challenging.

In a different direction, Chen \cite{Chen1971} established a Fenchel-Borsuk-Willmore type inequality for higher-codimensional submanifolds immersed in the Euclidean space. Specifically, he showed that if $\Sigma$ is an $n$-dimensional closed submanifold immersed in $\mathbb{R}^{n+m}$ and $\sigma$ is its normalized mean curvature vector, then
\begin{equation}\label{ineq chen}
\int_{\Sigma}|\sigma|^n \ge \left|\mathbb{S}^n\right|.
\end{equation}
Chen's proof of the Fenchel-Willmore inequality \eqref{ineq chen} relies on the Gauss map into the unit sphere, a construction specific to the Euclidean space that does not seem to be easily generalizable to more general ambient manifolds.

In light of the results discussed above, it is natural to ask whether a Fenchel-Willmore type inequality holds for submanifolds of arbitrary codimension immersed in a manifold with some kind of non-negative curvature. However, when the codimension is greater than one, the nonnegative Ricci curvature condition is insufficient to guarantee such an inequality. For example, the Eguchi-Hanson metric on $T \mathbb S^2$ is Ricci flat with Euclidean volume growth and has a totally geodesic submanifold $\mathbb S^2$ \cite[p. 270]{Anderson1990}, thus violating the inequality of the form $\int_\Sigma|\sigma|^2 \ge \theta |\mathbb S^2|$. Therefore, it is reasonable to consider extensions where the ambient manifold has nonnegative sectional curvature.

In this regard, we are able to prove the following sharp Fenchel-Willmore inequality for closed submanifolds\footnote{Throughout this paper, all manifolds and submanifolds are assumed to be smooth and orientable, but we do not require a submanifold to be connected.} of arbitrary dimension and codimension. In the codimension one case, our method offers a third proof, complementing those developed by Agostiniani-Fogagnolo-Mazzieri \cite{Agostiniani-Fogagnolo-Mazzieri} and by Wang \cite{Wang23}.

\begin{theorem}\label{thm general fenchel willmore}
Let $n, m \in \mathbb{N}$, and $(M, g)$ be a complete non-compact Riemannian manifold of dimension $n+m$ with nonnegative sectional curvature and asymptotic volume ratio $\theta>0$. Suppose $\Sigma$ is a closed $n$-dimensional submanifold immersed in $M$. Then
\begin{equation}\label{ineq fenchel willmore}
\int_{\Sigma}|\sigma|^n \ge \theta|\mathbb S^n|.
\end{equation}
If $m=1$, the curvature assumption may be relaxed to require only non-negative Ricci curvature.
\end{theorem}

Indeed, the same result remains valid under the weaker assumption of non-negative $k$-Ricci curvature for some $k$ depending on $n$ and $m$. To avoid disrupting the flow of exposition with technical details regarding this curvature condition, we defer the definition of the $k$-Ricci curvature and the related discussion to the end of the section.

To state the equality case of Theorem \ref{thm general fenchel willmore}, we need to introduce some notations. Let $\Sigma^+:=\{x\in \Sigma: \sigma(x)\ne0\}$.
On $\Sigma^{+}$, we decompose $T_x^{\perp} \Sigma$ into the orthogonal direct sum $\widetilde{T}_x^{\perp} \Sigma \oplus \operatorname{span}(\sigma)$, where $\widetilde{T}_x^{\perp} \Sigma$ is the orthogonal complement of $\sigma(x)$ in $T_x^{\perp} \Sigma$. In this case, $z \in T_x^{\perp} \Sigma$ can be written as $z=y+t \frac{\sigma}{|\sigma|}$, and we identify it with the pair $(y, t)$. We use the convention that for the unit sphere in the Euclidean space, $\sigma$ is the inward pointing unit normal.

We now characterize the equality case in Theorem \ref{thm general fenchel willmore}.
\begin{theorem} \label{thm equality case}
Assume either the curvature condition stated in Theorem \ref{thm general fenchel willmore}, or the weaker condition given in Theorem \ref{thm general fenchel willmore'}.

\begin{enumerate}
\item[(a)]
Suppose $n= 1$ and the equality in \eqref{ineq fenchel willmore} holds.
Then $\Sigma$ is connected and the normal exponential map when restricted to
$ \{(x, z)\in T^\perp\Sigma^{+}:\langle z, \sigma(x)\rangle\le0\}=
\{(x, y, t) \in \widetilde{T}^{\perp} \Sigma^{+} \times \mathbb{R}: t \le 0\}$
is a diffeomorphism onto its image. Moreover, the pullback of the ambient metric under the exponential map on this set is given by
\begin{equation}\label{eq g(x, y, t)}
g(x, y, t)=d t^2+d y^2+(1-t|\sigma(x)|)^2 g_{\Sigma}.
\end{equation}
\item[(b)]
Suppose $n\ge2$. Then the equality in \eqref{ineq fenchel willmore} holds if and only if all the conclusions in (a) are satisfied and moreover, $\Sigma$ is an embedded umbilical submanifold and $\sigma$ is parallel and nonzero everywhere.

Furthermore, with a change of variable $r=r_{0}-t$ with $r_0=\left(\frac{|\Sigma|}{\theta\left|\mathbb{S}^n\right|}\right)^{\frac{1}{n}}$,
the pullback metric in \eqref{eq g(x, y, t)} can be written as
$$g(x, y, r)=dr^{2}+dy^{2}+\left(\frac{r}{r_{0}}\right)^{2}g_{\Sigma}, $$
defined on $\widetilde T^\perp\Sigma\times[r_{0}, \infty)$.
\item [(c)]
Let $n \ge 1$ and suppose the equality in \eqref{ineq fenchel willmore} holds. For each $r>0$, let $\Sigma_r^{+}:=\{(x, z)\in T^{\perp}\Sigma^{+}:\langle z, \sigma(x)\rangle\le 0, |z|< r\}$
and $N_r(\Sigma):=\{p\in M:d(p, \Sigma)< r\}$ be the tubular neighborhood of $\Sigma$. Then
$$
\lim_{r\to\infty} \frac{|\exp(\Sigma_r^{+})|} {|N_r(\Sigma)|} =1.
$$
i.e. the image under the normal exponential map of normal vectors lying in the half-space in the direction of $-\sigma$ fills the neighborhood of $\Sigma$ with asymptotic volume density one.
\end{enumerate}
\end{theorem}

We remark that if a sharp constant is not required, a Fenchel-Willmore type inequality for a closed submanifold can already be derived by setting $f = 1$ in \cite[Theorem 1.4]{Brendle2023}, followed by an application of Holder's inequality, at least in the case where $M$ has non-negative sectional curvature.

In addition to the classical and more recent results on Fenchel-Willmore-type inequalities, this work is also inspired by the approach introduced by Brendle in his proof of Sobolev-type inequalities on manifolds with non-negative sectional curvature \cite{Brendle2023}. His approach is in turn partly motivated by techniques underlying the Alexandrov-Bakelman-Pucci maximum principle \cite{cabre2008elliptic}. It also has some similarities with the optimal transport proof of the isoperimetric inequality \cite{BrendleEichmair2023}. In \cite[Theorem 1.4]{Brendle2023}, Brendle constructs a ``transport map'' from a subset of the normal bundle of $\Sigma$ into the ambient manifold $M$ by using the solution to a linear elliptic equation on $\Sigma$ with a Neumann boundary condition (if $\partial \Sigma\ne\emptyset$). He then proves the surjectivity of this map onto suitable subsets of $M$, estimates its Jacobian by analyzing the elliptic equation together with an associated matrix Riccati equation along geodesics coming from $\Sigma$, and completes the argument via the change-of-variables formula for integrals.

We follow a similar strategy: constructing a ``transport map'', establishing its surjectivity onto suitable subsets, and applying Jacobian estimates in a change-of-variables formula to control the corresponding integrals. However, our approach differs in several respects. First, we avoid solving an auxiliary elliptic equation; our transport map is defined directly using the normal exponential map. Second, instead of the Alexandrov-Bakelman-Pucci principle and matrix Riccati methods, we prove a Heintze-Karcher type Jacobian comparison (which is already known) and a monotonicity formula for the Jacobian determinant (Lemma \ref{lem 2.4''}) by working with the scalar Riccati inequality, which lets us assume weaker curvature conditions. Third, in the higher-codimension setting, we carry out a more precise analysis of the fiber integrals arising from the volume estimate, as discussed below.

Brendle's Sobolev inequality (\cite[Theorem 1.4]{Brendle2023}) is sharp in the Euclidean setting for codimensions up to two. In higher codimensions, sharpness fails because one of the fiber‐integral estimates relies on the algebraic inequality $b^{\frac{m}{2}}-a^{\frac{m}{2}} \le \frac{m}{2}(b-a)$ for $0 \le a< b \le 1$ and $m\ge2$, which is no longer optimal when $m>2$.

We overcome this difficulty by carrying out a more precise evaluation of the relevant fiber integrals, rather than bounding them, using a representation formula \eqref{eq euler} related to the so-called Euler integral of the first kind, expressed in terms of the gamma function $\Gamma$. Integrals of the form $\int_0^1 t^p (1 - t^2)^q \, dt$ appear repeatedly in our analysis, both in the estimation of fiber integrals in the proof of the Fenchel-Willmore inequality and in the proof of rigidity. Specifically, we compute the fiber integrals (up to error terms that vanish in the limiting process) using this representation formula. The use of the gamma function in this context allows us to obtain rather precise expressions, thereby preserving the sharp constant for all codimensions $m$ and $\theta > 0$.

Moreover, the equality case in Brendle's result is achieved only when $\Sigma$ is a flat disk (with boundary) of codimension up to two in Euclidean space, whereas in our setting, equality holds only if $\Sigma$ is closed, umbilical, and with constant $|\sigma|>0$ (if $n>1$), with $\theta$ possibly less than $1$. Consequently, our analysis must proceed along a different path to address this difference. For example, our argument requires separate treatment on the set where the mean curvature vector vanishes and where it is non-zero, and our equality case analysis will accordingly focus on the region where $\sigma\ne 0$.

We now return to the earlier remark concerning a weaker curvature assumption. It turns out that the curvature condition stated in Theorem~\ref{thm general fenchel willmore} can be substantially relaxed, especially when $n$ and $m$ are large or when $m=1$, to a condition involving the so-called $k$-Ricci curvature.

We begin by recalling the definition of $k$-Ricci curvature (also referred to as intermediate Ricci curvature) on a Riemannian manifold $(M, g)$. Let $v \in T_pM$ be a unit tangent vector, and let $W \subset T_pM$ be a $k$-dimensional subspace orthogonal to $v$. The $k$-Ricci curvature in the direction of $v$ with respect to $W$ is defined as
$$\overline{\mathrm{Ric}}_k(v, W) = \sum_{i=1}^{k} \langle \bar{R}(v, e_i)e_i, v \rangle, $$
where $\{e_i\}_{i=1}^{k}$ is any orthonormal basis of $W$ and $\bar R$ is the Riemannian curvature tensor. This notion interpolates between sectional curvature and Ricci curvature: the case $k=1$ recovers sectional curvature, while $k = \dim M - 1$ gives the Ricci curvature in the direction of $v$.

We say $(M, g)$ has non-negative $k$-Ricci curvature, denoted $\overline{\mathrm{Ric}}_k \ge 0$, if $\overline{\mathrm{Ric}}_k(v, W) \ge 0$ at every point $p \in M$, for every unit vector $v \in T_pM$, and every $k$-plane $W \subset T_pM$ orthogonal to $v$. Clearly, the cases $k=1$ and $k=\dim M-1$ recover non-negative sectional curvature and non-negative Ricci curvature, respectively. It is straightforward to verify that $\overline{\mathrm{Ric}}_k \ge 0$ implies $\overline{\mathrm{Ric}}_l \ge 0$ for all $l > k$.

Lower bounds on the intermediate $k$-Ricci curvature have been used to bridge the gap between results based on sectional curvature and those based on Ricci curvature in the earlier works of Wu \cite{Wu1979}, Shen \cite{Shen1993}, and Shen-Wei \cite{ShenWei1993}. Since then, many results related to volume growth, diameter bounds, Betti numbers, geometric flows, optimal transport, and other geometric and topological applications have been obtained under a lower bound on this curvature. See, for instance, \cites{BrendleHuisken2017, GuijarroWilhelm2018, Chahine2019, KettererMondino2018}. For more details, a comprehensive list of articles on intermediate $k$-Ricci curvature can be found on the webpage \cite{IntermediateRicciCurvatureList}.

We are able to weaken the curvature assumption in Theorem \ref{thm general fenchel willmore} as follows.
\begin{theorem}\label{thm general fenchel willmore'}
The conclusion of Theorem \ref{thm general fenchel willmore} still holds under the weaker assumption that the $k$-Ricci curvature is non-negative, where
\begin{equation}\label{eq k}
k=
\begin{cases}
\min(n, m-1), & \text{if } m > 1, \\
n, & \text{if } m = 1.
\end{cases}
\end{equation}
Note that in both cases, the Ricci curvature of $M$ is non-negative and $\theta$ is still well-defined.
\end{theorem}

As an immediate consequence, we have
\begin{corollary}
There is no $n$-dimensional closed minimal submanifold immersed in a complete Riemannian manifold $(M^{n+m}, g)$ with $\overline{\mathrm{Ric}}_k \ge 0$ and Euclidean volume growth, where $k$ is given by \eqref{eq k}.
\end{corollary}

The weakening of the curvature assumption to a $k$-Ricci condition is made possible by a more refined analysis of Jacobian determinant via the Riccati inequality, which strengthens the classical Jacobian comparison theorem of Heintze and Karcher \cite{HeintzeKarcher1978}. While a comparison result involving $k$-Ricci curvature is known (see \cite{Chahine2019}), we are not aware of any reference establishing the monotonicity of the Jacobian determinant under this weaker assumption (Lemma \ref{lem 2.4''}). As this monotonicity plays a crucial role in the rigidity argument, we provide the full details here.

Moreover, the non-negative $k$-Ricci curvature condition also proves sufficient for characterizing umbilicity, and for establishing that the mean curvature vector is parallel when $n > 1$ by invoking the Codazzi equation.

The remainder of this paper is organized as follows. In Section \ref{sec willmore fenchel}, we present the proof of the Fenchel-Willmore inequality, Theorem \ref{thm general fenchel willmore'}, under the non-negative $k$-Ricci curvature assumption. In Section \ref{sec equality}, we begin by presenting examples of $\Sigma$ and $M$ that attain the equality, and then proceed to analyze the equality case and prove Theorem \ref{thm equality case}.

\noindent\textbf{Acknowledgements.}
The second named author is grateful to Ben Andrews, Jihye Lee and Yong Wei for valuable discussions. He was supported by the University of Wollongong Early-Mid Career Researcher Enabling Grant and the UOW Advancement and Equity Grant Scheme for Research 2024.

\section{Proof of the Fenchel-Willmore inequality}\label{sec willmore fenchel}
In this paper, we consider a complete, noncompact Riemannian manifold $(M, g)$ of dimension $n + m$, together with a closed immersed submanifold $\Sigma \subset M$ of dimension $n$. The Levi-Civita connection of $(M, g)$ is denoted by $\bar{D}$, and its Riemann curvature tensor by $\bar{R}$, with the convention $\bar{R}(X, Y)Z = \bar{D}_X \bar{D}_Y Z - \bar{D}_Y \bar{D}_X Z - \bar{D}_{[X, Y]} Z$. The second fundamental form of $\Sigma$, denoted by $\mathrm{II}$, is a symmetric bilinear form on the tangent bundle of $\Sigma$ that takes values in the normal bundle $T^\perp \Sigma$. At a point $x \in \Sigma$, for tangent vector fields $X, Y$ and a normal vector field $V$, the second fundamental form satisfies $\langle \mathrm{II}(X, Y), V \rangle = \langle \bar{D}_X Y, V \rangle$. The normalized mean curvature vector of $\Sigma$ is given by $\sigma = \frac{1}{n} \operatorname{tr} \mathrm{II}$.

From now on, we only assume the weaker curvature assumption stated in Theorem \ref{thm general fenchel willmore'}.

We define $\Phi_r(x, z):=\exp_{x}\left(r z\right)$ for $(x, z)\in T^\perp \Sigma$ and $r>0$, and
\begin{equation*}
\begin{aligned}
& A_r:= \{ (x, z) \in T^\perp \Sigma: |z|< 1 \text{ and } d\left(q, \exp_x\left(r z\right)\right) \ge r |z| \text{ for all } q \in \Sigma \}.
\end{aligned}
\end{equation*}
\begin{lemma}\label{lem 2.2''}
For every $0 <\alpha<1$ and $r>0$,
$$
\{p \in M: \alpha r<d(x, p)<r \text { for all } x \in \Sigma\}
$$
is contained in $\left\{\Phi_r(x, z):(x, z) \in A_r \text { and }|z|>\alpha\right\}$.
\end{lemma}
\begin{proof}
Fix a real number $0 <\alpha<1$ and a point $p \in M$ such that $\alpha r<d(x, p)<r$ for all $x \in \Sigma$.
Suppose the distance function $d(\cdot, p)$ attains its minimum at a point $\bar{x} \in \Sigma$.
Let $\bar{\gamma}:[0, r] \rightarrow M$ be a minimizing geodesic such that $\bar{\gamma}(0)=\bar{x}$ and $\bar{\gamma}(r)=p$.
Then $\bar z:=\bar\gamma'(0)\in T_{\bar x}^\perp \Sigma$ by the first variation formula. Clearly, $r|\bar z|=r\left|\bar{\gamma}^{\prime}(0)\right|=d(\bar{x}, p)\in(\alpha r, r)$, and
$\Phi_r(\bar{x}, \bar{z})=\exp_{\bar{x}}\left(r \bar{z}\right)=\exp_{\bar{\gamma}(0)}\left(r \bar{\gamma}^{\prime}(0)\right)=\bar{\gamma}(r)=p$.
For any $x\in \Sigma$, $d(x, p)\ge d(\bar x, p)= r|z|$, i.e. $(\bar x, \bar z)\in A_r$.
\end{proof}

\begin{lemma}\label{lem 2.3''}
For every $(\bar x, \bar z) \in A_r$, we have $1-r \langle \sigma(\bar x), \bar z\rangle \ge 0$.
\end{lemma}
\begin{proof}
Let $\bar{\gamma}(s):=\exp_{\bar{x}}\left(s \bar{z}\right)$ for $s \in[0, r]$.
Let $\gamma:[0, r] \rightarrow M$ be an arbitrary smooth path satisfying $\gamma(0) \in \Sigma$ and $\gamma(r)=\bar{\gamma}(r)$. Since $(\bar{x}, \bar{z}) \in A_r$, we have
\begin{align*}
r \int_0^r\left|\gamma^{\prime}(s)\right|^2 d s \ge d(\gamma(0), \gamma(r))^2
=d\left(\gamma(0), \exp_{\bar{x}}\left(r \bar{z}\right)\right)^2
\ge r^2 |\bar{z}|^2
=r \int_0^r\left|\bar{\gamma}^{\prime}(s)\right|^2 d s.
\end{align*}

In other words, $\bar{\gamma}$ minimizes the energy functional $\int_0^r\left|\gamma^{\prime}(s)\right|^2 d s$ among all smooth paths $\gamma:[0, r] \rightarrow M$ satisfying $\gamma(0) \in \Sigma$ and $\gamma(r)=\bar{\gamma}(r)$. Let $Z$ be a smooth vector field along $\bar{\gamma}$ such that $Z(0) \in T_{\bar{x}} \Sigma$ and $Z(r)=0$.
By the second variation formula of energy, we obtain
\begin{align*}
-\left\langle \mathrm{II}(Z(0), Z(0)), \bar{\gamma}^{\prime}(0)\right\rangle +\int_0^r\left(\left|\bar{D}_s Z(s)\right|^2-\left\langle \bar{R}\left(\bar{\gamma}^{\prime}(s), Z(s)\right) Z(s), \bar{\gamma}^{\prime}(s) \right\rangle
\right) d s \ge 0.
\end{align*}
Take an orthonormal basis $\{e_i\}_{i=1}^n$ of $T_{\bar x}\Sigma$ and let $E_i$ be the parallel transport of $e_i$ along $\bar{\gamma}$. Applying the above inequality to the vector fields $Z_i(s):=(r-s) E_i(s)$ and summing from $i=1$ to $n$ gives
\begin{equation*}
nr-r^2\langle n \sigma(\bar x), \bar{z}\rangle
-\sum_{i=1}^n\int_0^r(r-s)^2\left\langle \bar{R}\left(\bar{\gamma}^{\prime}(s), E_i(s)\right) E_i(s), \bar{\gamma}^{\prime}(s) \right\rangle
d s \ge 0.
\end{equation*}
In view of the assumption $\overline{\mathrm{Ric}}_k \ge 0$, with $k$ as defined in \eqref{eq k}, which implies $\overline{\mathrm{Ric}}_n \ge 0$, the desired result follows.
\end{proof}

The following form of the Riccati comparison theorem can be found in, for example, \cite[Lemma 4.1]{BallmannRiccati}.
\begin{lemma}\label{lem riccati}
Let $q_1, q_2:(0, a] \rightarrow \mathbb{R}$ be smooth functions
such that
$$ q_1'(s)+q_1(s)^2 \le q_2'(s)+q_2(s)^2 $$
and both $q_i'(s)+q_i(s)^2$ extend smoothly to $[0, a]$.
If $\lim_{s \rightarrow 0^+} (q_2(s)-q_1(s)) =0$,
then $q_1(s) \le q_2(s)$.
If $q_1(a)=q_2(a)$, then $q_1=q_2$ on $(0, a]$.
\end{lemma}

\begin{lemma}\label{lem 2.4''}
Assume that $(x, z) \in A_r$. Then the function
$s \mapsto \frac{|\det D\Phi_s(x, z)|}{s^m(1-s\langle \sigma(x), z\rangle)^n}$
is monotone decreasing and
\begin{align*}
|\det D\Phi_s(x, z)|\le {s^m(1-s\langle \sigma, z\rangle)^n}.
\end{align*}
for $s \in(0, r)$.
\end{lemma}

\begin{proof}
The estimate for the Jacobian determinant actually follows from \cite[Corollary 3.3.1]{HeintzeKarcher1978} (or \cite[Corollary 3.3.2]{HeintzeKarcher1978} in the case when $m=1$). See also \cite{Chahine2019}. But since we also need the monotonicity formula, we will carry out the complete proof here.

Take any $r>0$ and $(\bar{x}, \bar{z})\in A_r$.
Define the geodesic
$$
\bar{\gamma}(s)=\exp_{\bar{x}}(s\bar{z}), \qquad s\in[0, r].
$$

Choose an orthonormal basis $\{e_1, \dots, e_n\}$ of $T_{\bar{x}}\Sigma$ such that the matrix $\left(-\langle \mathrm{II}(e_i, e_j), \bar{z}\rangle\right)=\Lambda=\mathrm{diag}(\lambda_1, \cdots, \lambda_n)$ is diagonal.
Select an orthonormal frame $\{e_{n+1}, \dots, e_{n+m}\}$ of $T^\perp\Sigma$ near $\bar{x}$ with $\langle\bar D_{e_i}e_\alpha, e_\beta\rangle=0$ at $\bar{x}$ for all $1\le i\le n$ and $n+1\le\alpha, \beta\le n+m$, such that $e_{n+m}$ is parallel to $\bar z$ and $\{e_1, \cdots, e_{n+m}\}$ is positively oriented.
Let $E_j(s)$ be the parallel transport of $e_j$ along $\bar{\gamma}$ for $j=1, \cdots, n+m$.

For each $1\le i\le n$, we define the Jacobi field $X_i(s)$ along $\bar{\gamma}$ by
$$ X_i(0)=e_i \text{ and }X_i'(0)
=\lambda_i e_i.
$$

For each $n+1 \le \alpha \le n+m$, we denote by $X_\alpha(s)$ the Jacobi field along $\bar{\gamma}$ satisfying
\begin{equation*}
X_\alpha(0)=0\text{ and }\bar{D}_s X_\alpha(0)=e_\alpha.
\end{equation*}
Define the matrix $P(s)$ by setting $P_{ij}(s)=\langle E_j(s), X_i(s)\rangle$ and $S_{ij}(s)=\langle \bar{R} (E_i(s), \bar{\gamma}'(s)) \bar{\gamma}'(s), E_j(s)\rangle$. Then $P$ satisfies the Jacobi equation $P''(s)=-S(s)P(s)$.

From the definition of $A_r$, there is no focal point of $\Sigma$ along $\bar\gamma(s)$ for $0<s<r$. In particular, $\{X_j\}_{j=1}^{n+m}$ is linearly independent and $P(s)$ is invertible. Define $Q(s)=P'(s)P^{-1}(s)$, then $Q(s)$ is symmetric and satisfies the Riccati equation
$$
Q'(s)+Q(s)^2=-S(s).
$$

We remark that $Q(s)$ can be interpreted as the matrix representation of $|\bar z| B(s)$ with respect to $\{E_j\}_{j=1}^{n+m}$, where $B(s)=\bar D^2 \rho$ is the shape operator of $\Sigma_s$ at $\bar\gamma(s)$ and $\Sigma_s$ is the hypersurface whose distance $\rho$ from $\Sigma$ is $|\bar z|s$.

Define the normalized partial traces
$$
q_n(s)=\frac{1}{n}\sum_{i=1}^n Q_{ii}(s), \quad
q_m(s)=\frac{1}{m}\sum_{\alpha=n+1}^{n+m}Q_{\alpha\alpha}(s).
$$
By the Cauchy-Schwarz inequality, they satisfy the scalar Riccati inequality
$$
q_n'(s)+ q_n(s)^2\le
q_n'(s)+\frac{1}{n}\sum_{i=1}^{n}\left(Q(s)^2\right)_{ii} = -\frac{1}{n}\sum_{i=1}^{n}\left\langle\bar{R}\left(E_i(s), \bar{\gamma}^{\prime}(s)\right) \bar{\gamma}^{\prime}(s), E_i(s)\right\rangle\le0
$$
and
$$
q_m^{\prime}(s)+ q_m(s)^2 \le q_m^{\prime}(s)+\frac{1}{m}\sum_{\alpha =n+1}^{ n+m}(Q(s)^2)_{ii} = -\frac{1}{m}\sum_{\alpha =n+1}^{n+m}\langle
\bar{R}(E_\alpha (s), \bar{\gamma}^{\prime}(s))\bar{\gamma}^{\prime}(s), E_\alpha (s) \rangle \le0.
$$
More precisely, the first inequality comes from
\begin{equation}\label{ineq cs}
\begin{aligned}
\frac{1}{n} \sum_{i=1}^n\left(Q(s)^2\right)_{i i}=\frac{1}{n} \sum_{i=1}^n Q_{i i}(s)^2+\frac{1}{n} \sum_{i=1}^n \sum_{j \ne i} Q_{i j}(s)^2 \ge & \frac{1}{n} \sum_{i=1}^n Q_{i i}(s)^2\\
\ge& \left(\frac{1}{n} \sum_{i=1}^n Q_{i i}(s)\right)^2\\
=& q_n(s)^2.
\end{aligned}
\end{equation}
Notice also that the curvature assumption $\overline{\mathrm{Ric}}_k\ge0$ with $k$ given by \eqref{eq k} implies both
$$\sum_{i=1}^n\left\langle\bar{R}\left(E_i(s), \bar{\gamma}^{\prime}(s)\right) \bar{\gamma}^{\prime}(s), E_i(s)\right\rangle\ge0$$ and
$$\sum_{\alpha=n+1}^{n+m}\left\langle\bar{R}\left(E_\alpha(s), \bar{\gamma}^{\prime}(s)\right) \bar{\gamma}^{\prime}(s), E_\alpha(s)\right\rangle=
\sum_{\alpha=n+1}^{n+m-1}\left\langle\bar{R}\left(E_\alpha(s), \bar{\gamma}^{\prime}(s)\right) \bar{\gamma}^{\prime}(s), E_\alpha(s)\right\rangle\ge0. $$

As shown in \cite[Equation 2.4]{Chahine2019}, by computing the Taylor expansion of $\bar D^2 \rho$, where $\rho$ is the distance from $\Sigma$, it holds that
\begin{align}\label{Q0}
Q(s)=
\begin{bmatrix}
\Lambda+O(s) & O(s)\\
O(s) & \frac{1}{s}I_m+O(s)
\end{bmatrix}
\end{align}
as $s\to0^+$.

From this, we have
\begin{align*}
q_n(s)=-\langle \sigma, \bar z\rangle +O(s)\quad \text{ and } q_m(s)=\frac{1}{s} +O(s)
\end{align*}
as $s\to0^+$.

Let $\bar q_n(s)=-\frac{\langle \sigma, \bar z\rangle }{1-s\langle \sigma, \bar z\rangle}$ and $\bar q_m(s)=\frac{1}{s}$, then they both satisfy the scalar Riccati equation $f'(s)+f(s)^2=0$. In particular,
\begin{align*}
q_n^{\prime}(s)+q_n(s)^2 \le \bar q_n^{\prime}(s)+\bar q_n(s)^2\quad \text{and}\quad q_m^{\prime}(s)+q_m(s)^2 \le \bar q_m^{\prime}(s)+\bar q_m(s)^2.
\end{align*}

By Lemma \ref{lem riccati}, we conclude that
\begin{align*}
q_n(s) \le\bar q_n(s)\quad \text{and}\quad q_m(s) \le\bar q_m(s).
\end{align*}
We then have
$$
\frac{d}{d s} \log (\det P(s))=\operatorname{tr} Q(s)
=n q_n(s) +m q_m(s) \le n\bar q_n(s) +m \bar q_m(s)=\frac{m}{s}-\frac{n\langle\sigma, \bar{z}\rangle}{1-s\langle\sigma, \bar{z}\rangle}.
$$

From the observation that $\det D\Phi_s =\det P(s)$, we deduce that
\begin{align*}
\frac{|\det D\Phi_s|}{s^m(1-s\langle \sigma, \bar z\rangle)^n}
\end{align*}
is non-increasing, and so by the initial condition
\begin{equation*}
P(s)=\begin{bmatrix}
I_n+ O(s) & O(s) \\
O(s) & s I_m+O(s^2)
\end{bmatrix},
\end{equation*}
we can then integrate this inequality to get
\begin{align*}
|\det D\Phi_s|\le {s^m(1-s\langle \sigma, \bar z\rangle)^n}.
\end{align*}
\end{proof}
We are now ready to prove Theorem \ref{thm general fenchel willmore'}.

\begin{proof}[Proof of Theorem \ref{thm general fenchel willmore'}]
Let $\alpha\in(0, 1)$ and $r>0$. Define $\Sigma^+=\{x\in \Sigma: \sigma(x)\ne 0\}$ and $\Sigma^0=\{x\in \Sigma: \sigma(x) =0\}$.

By Lemma \ref{lem 2.2''},
\begin{equation}\label{ineq est I0 I+}
\begin{aligned}
& |\{p \in M: \alpha r<d(x, p)<r \text { for all } x \in \Sigma\} | \\
\le & \int_{\Sigma} \int_{Z_\alpha}\left|\det D\Phi_r(x, z)\right| 1_{A_r}(x, z) dz d \operatorname{vol}_{\Sigma} \\
=& \int_{\Sigma^0} \int_{Z_\alpha}\left|\det D\Phi_r(x, z)\right| 1_{A_r}(x, z) dz d \operatorname{vol}_{\Sigma}
+ \int_{\Sigma^+} \int_{Z_\alpha}\left|\det D\Phi_r(x, z)\right| 1_{A_r}(x, z) dz d \operatorname{vol}_{\Sigma}\\
=& I^0(\alpha, r)+I^+(\alpha, r),
\end{aligned}
\end{equation}
where $Z_\alpha=\{z\in T_x^\perp \Sigma: \alpha<|z|<1\}$.

Note that
\begin{equation*}
\lim_{r\to\infty}\frac{1}{r^{n+m}}|\{p \in M: \alpha r<d(x, p)<r \text { for all } x \in \Sigma\} |=\theta\left|\mathbb{B}^{n+m}\right|\left(1-\alpha^{n+m}\right)
\end{equation*}
and hence
\begin{equation}\label{avr}
\begin{split}
\lim_{\alpha\to1^-}\lim_{r\to\infty}\frac{1}{1-\alpha}\cdot\frac{1}{r^{n+m}}|\{p \in M: \alpha r<d(x, p)<r \text { for all } x \in \Sigma\} |=(n+m)\theta\left|\mathbb{B}^{n+m}\right|.
\end{split}
\end{equation}
Let us assume $m \ge 3$ at the moment.

We first consider the integral $I^+(\alpha, r)$.
As explained in the Introduction (Section \ref{sec intro}), on $\Sigma^{+}$ we decompose $T_x^{\perp} \Sigma=\widetilde{T}_x^{\perp} \Sigma \oplus \operatorname{span}(\sigma)$, so any normal vector can be written as $y+t \frac{\sigma}{|\sigma|}$ with $y \in \widetilde{T}_x^{\perp} \Sigma$, and we identify it with $(y, t)$.
Lemma \ref{lem 2.3''} implies that
$t$ satisfies $-1<t \le \frac{1}{r|\sigma(x)|}$ for $(x, y, t)\in A_r$, so Lemmas \ref{lem 2.2''} and \ref{lem 2.4''} give
\begin{equation}\label{ineq I+}
\begin{aligned}
I^+(\alpha, r) & \le \int_{\Sigma^+} \int_{-1}^{\frac{1}{r|\sigma(x)|}} \int_{Y_{\alpha, x, t}}\left|\det D \Phi_r(x, y, t)\right| 1_{A_r}(x, y, t) d y d t d \operatorname{vol}_{\Sigma} (x)\\
& \le \int_{\Sigma^+} \int_{-1}^{\frac{1}{r|\sigma|}} \int_{Y_{\alpha, x, t}} r^m(1-r|\sigma(x)| t)^n d y d t d \operatorname{vol}_{\Sigma},
\end{aligned}
\end{equation}
where $Y_{\alpha, x, t}:=\left\{y \in \widetilde{T}_x^{\perp} \Sigma: \alpha^2<|y|^2+t^2<1\right\}$.

For $|t|<1$, we have the limit
$$
\lim_{\alpha \rightarrow 1^{-}} \frac{\left(1-t^2\right)^{\frac{m-1}{2}} -\left(\alpha^2-t^2\right)^{\frac{m-1}{2}} }{1-\alpha}=(m-1) \left(1-t^2\right)^{\frac{m-3}{2}}.
$$
Therefore, for $(x, y, t) \in A_r$, as $\alpha\to 1^-$ and $r\to\infty$,
\begin{align*}
\left|Y_{\alpha, x, t}\right| & =\left|\mathbb{B}^{m-1}\right|\left(\left(1 -t^2\right)_{+}^{\frac{m-1}{2}}-\left(\alpha^2 -t^2\right)_{+}^{\frac{m-1}{2}}\right) \\
& =(m-1)\left|\mathbb{B}^{m-1}\right|\left(\left(1-t^2\right)^{\frac{m-3}{2}}(1-\alpha)+O\left((1-\alpha)^2\right)\right).
\end{align*}
Here $O\bigl((1-\alpha)^p\bigr)$ is a quantity whose absolute value is bounded by $C (1-\alpha)^p$ for some constant $C > 0$ independent of $\alpha$, when $\alpha$ is sufficiently close to $1$.

Hence for $r$ large enough and $\alpha$ close enough to $1$, the fiber integral in \eqref{ineq I+} satisfies
\begin{equation}\label{ineq I+2}
\begin{aligned}
& \int_{-1}^{\frac{1}{r|\sigma|}} \int_{Y_{\alpha, x, t}} r^m(1-r|\sigma| t)^n d y d t\\
=& (m-1)\left|\mathbb{B}^{m-1}\right|r^m
\int_{-1}^{\frac{1}{r|\sigma|}} \left(\left(1-t^2\right)^{\frac{m-3}{2}}(1-\alpha)+O\left((1-\alpha)^2\right)\right)(1-r|\sigma| t)^n d t.
\end{aligned}
\end{equation}

Now, recall the identity \cite[Eqn. (13.45), (13.49)]{ArfkenWeberHarris2013}
\begin{equation}\label{eq euler}
\int_0^1 s^{p-1}(1-s)^{q-1} d s=\frac{\Gamma(p) \Gamma(q)}{\Gamma(p+q)},
\end{equation}
where $p, q>0$ and $\Gamma(s)=\int_0^{\infty} t^{s-1} e^{-t} d t$ is the gamma function.

For $r$ large enough, we consider
\begin{equation}\label{ineq I+3}
\begin{aligned}
\int_{-1}^{\frac{1}{r|\sigma|}} \left(1-t^2\right)^{\frac{m-3}{2}} (1-r|\sigma| t)^n d t
=& r^n|\sigma|^n \int_{-1}^0(-t)^n\left(1-t^2\right)^{\frac{m-3}{2}} d t+O\left(r^{n-1}\right)\\
=& \frac{\Gamma\left(\frac{n+1}{2}\right) \Gamma\left(\frac{m-1}{2}\right)}{2\Gamma\left(\frac{n+m}{2}\right)} r^n|\sigma|^n+O\left(r^{n-1}\right).
\end{aligned}
\end{equation}
(Here $O\bigl(r^p\bigr)$ is a quantity whose absolute value is bounded by $C r^p$ for some constant $C > 0$ independent of $r$, when $r$ is large enough.)

On the other hand,
\begin{equation}\label{ineq I+4}
\int_{-1}^{\frac{1}{r|\sigma|}}(1-r|\sigma| t)^n d t=
\frac{(1+r|\sigma|)^{n+1}}{(n+1) r|\sigma|}=O(r^n).
\end{equation}

For sufficiently large $r$ and $\alpha$ close to $1$, \eqref{ineq I+}, \eqref{ineq I+2}, \eqref{ineq I+3} and \eqref{ineq I+4}, together with the identity $\left|\mathbb{B}^d\right| = \frac{\pi^{\frac{d}{2}}}{\Gamma\left(\frac{d}{2} + 1\right)}$ and the relation $\Gamma(z + 1) = z\, \Gamma(z)$, yield
\begin{equation*}
\begin{aligned}
\frac{1}{r^{m+n}}I^{+}(\alpha, r)
\le& (1-\alpha)(m-1)\left|\mathbb{B}^{m-1}\right|\frac{\Gamma\left(\frac{n+1}{2}\right) \Gamma\left(\frac{m-1}{2}\right)}{2 \Gamma\left(\frac{n+m}{2}\right)}\int_{\Sigma^+}|\sigma|^n+O((1-\alpha)^2)\\
=& (1-\alpha)\pi^{\frac{m-1}{2}} \frac{\Gamma\left(\frac{n+1}{2}\right)}{\Gamma\left(\frac{n+m}{2}\right)} \int_{\Sigma^+} |\sigma|^n +O((1-\alpha)^2).
\end{aligned}
\end{equation*}
Therefore
\begin{equation}\label{ineq est I+}
\limsup_{\alpha \rightarrow 1^{-}} \limsup_{r \rightarrow \infty} \frac{1}{1-\alpha} \cdot \frac{1}{r^{n+m}} I^{+}(\alpha, r) \le
\pi^{\frac{m-1}{2}} \frac{\Gamma\left(\frac{n+1}{2}\right)}{\Gamma\left(\frac{n+m}{2}\right)} \int_{\Sigma}|\sigma|^n.
\end{equation}
We now turn to $I^0(\alpha, r)$. In this case, by Lemma \ref{lem 2.4''} we have
\begin{equation*}
\begin{aligned}
I^0(\alpha, r)=& \int_{\Sigma^0} \int_{Z_\alpha}\left|\det D\Phi_r(x, z)\right| 1_{A_r}(x, z) d z d \operatorname{vol}_{\Sigma}(x)\\
\le& \int_{\Sigma^0} \int_{Z_\alpha}r^m d z d \operatorname{vol}_{\Sigma}(x)\\
=& |\mathbb B^m|(1-\alpha^m)r^m |\Sigma^0|,
\end{aligned}
\end{equation*}
and so
$$
\lim_{\alpha \rightarrow 1^{-}} \lim_{r \rightarrow \infty} \frac{1}{1-\alpha} \cdot \frac{1}{r^{n+m}} I^0(\alpha, r)=0.
$$
Combining this with \eqref{ineq est I0 I+}, \eqref{avr} and \eqref{ineq est I+}, we obtain
$$
\theta(n+m)\left|\mathbb{B}^{n+m}\right| \le \pi^{\frac{m-1}{2}} \frac{\Gamma\left(\frac{n+1}{2}\right)}{\Gamma\left(\frac{n+m}{2}\right)} \int_{\Sigma^{+}}|\sigma|^n.
$$
This inequality is of the form $\int_{\Sigma}|\sigma|^n \ge \theta C$ for some positive constant $C$.

Using $|\mathbb B^{d}|=\frac{\pi^{\frac{d}{2}}}{\Gamma\left(\tfrac d2+1\right)}$ and
$|\mathbb S^{d-1}|=d|\mathbb B^{d}|=\frac{2\pi^{\frac{d}{2}}}{\Gamma\left(\tfrac d2\right)}$, we compute $C$ to be
$$
\frac{(n+m) |\mathbb B^{n+m}|}{
\pi^{ \frac{m-1}{2}} \Gamma\bigl(\tfrac{n+1}{2}\bigr)/ \Gamma\bigl(\tfrac{n+m}{2}\bigr)}
=(n+m) \frac{\pi^{\frac{n+m}{2}}}{\pi^{\frac{m-1}{2}}} \frac{\Gamma(\tfrac{n+m}{2})}{ \Gamma(\tfrac{n+m}{2}+1)\Gamma(\tfrac{n+1}{2})}
=\frac{2\pi^{\frac{n+1}{2}}}{ \Gamma(\tfrac{n+1}{2})}
=|\mathbb S^{n}|.
$$
From this, we arrive at
\begin{align*}
\int_\Sigma |\sigma|^n\ge \theta|\mathbb S^n|.
\end{align*}
Let us now consider the case where $m=1$ or $2$.

In the case where $m=2$, note that the condition in Theorem \ref{thm general fenchel willmore'} is that the sectional curvature of $M$ is non-negative. By taking the product of $M$ with $\mathbb{R}$, we can view $\Sigma$ as a codimension $3$ submanifold. Observe that the product manifold $M \times \mathbb{R}$ still has non-negative sectional curvature, while both $\theta$ and $\int_{\Sigma}|\sigma|^n$ remain invariant. Consequently, the inequality $\int_\Sigma |\sigma|^n \ge \theta|\mathbb S^{n}|$ continues to hold.

The remaining case is when $m=1$. Note that it is no longer true that the condition $\overline{\operatorname{Ric}}_n \ge 0$ is preserved in the product $M \times \mathbb{R}^2$. Nevertheless, the original argument remains valid after a minor modification. In this case, \eqref{ineq est I0 I+} and \eqref{avr} remain the same. We can modify \eqref{ineq I+} as follow:
\begin{equation*}\label{ineq I+'}
\begin{aligned}
I^{+}(\alpha, r) & \le \int_{\Sigma^{+}} \int_{T_{\alpha, r, x}} \left|\det D \Phi_r(x, t)\right| 1_{A_r}(x, t)\, d t\, d \operatorname{vol}_{\Sigma}(x) \\
& \le \int_{\Sigma^{+}} \int_{T_{\alpha, r, x}} r(1-r|\sigma(x)| t)^n\, d t\, d \operatorname{vol}_{\Sigma}(x),
\end{aligned}
\end{equation*}
where $T_{\alpha, r, x}=\{t\in \mathbb R: -1<t<\frac{1}{r|\sigma(x)|}, \alpha<|t|<1\}$.

For fixed $x\in \Sigma^+$, and for $r$ sufficiently large and $\alpha$ sufficiently close to $1$, consider the fiber integral
\begin{align*}
\int_{T_{\alpha, r, x}} r(1-r|\sigma(x)| t)^n d t
=& r\int_{-1}^{-\alpha} (1-r|\sigma(x)| t)^n d t\\
=& r^{n+1}|\sigma(x)|^n \int_\alpha^1 t^n d t+O\left(r^n\right)\\
=& \frac{r^{n+1}|\sigma(x)|^n}{n+1}\left(1-\alpha^{n+1}\right)+O\left(r^n\right).
\end{align*}
Dividing both sides by $r^{n+1}$ and $1-\alpha$, we then obtain the following inequality as a replacement for \eqref{ineq est I+}:
\begin{equation*}\label{ineq est I+'}
\limsup_{\alpha \rightarrow 1^{-}} \limsup_{r \rightarrow \infty} \frac{1}{1-\alpha} \cdot \frac{1}{r^{n+1}} I^{+}(\alpha, r) \le
\int_{\Sigma^+}|\sigma|^n.
\end{equation*}
The same limiting relation remains valid when $m=1$:
$$
\lim_{\alpha \rightarrow 1^{-}} \lim_{r \rightarrow \infty} \frac{1}{1-\alpha} \cdot \frac{1}{r^{n+1}} I^0(\alpha, r)=0.
$$
We can proceed as before to obtain
\begin{align*}
\int_{\Sigma}|\sigma|^n\ge \theta(n+1)\left|\mathbb B^{n+1}\right|=\theta |\mathbb S^n|.
\end{align*}
\end{proof}
\section{Equality case}\label{sec equality}

In this section, we analyze the equality case in Theorem \ref{thm general fenchel willmore'}. First of all, it is easy to see that the inequality is sharp. For example, take $\Sigma$ to be a closed $n$-dimensional manifold with positive curvature and define the conical metric
$$g = dr^2 + \left(\frac{r}{r_0} \right)^2 g_\Sigma$$
on $[r_0, \infty) \times \Sigma$, where $r_0 = \left(\frac{|\Sigma|}{\theta |\mathbb{S}^n|} \right)^{\frac{1}{n}}$.
By gluing with a suitable compact set along the boundary, the metric can be extended to a complete $n+1$ dimensional manifold $M$ with non-negative curvature. One then finds that the asymptotic volume ratio is $\theta$, and $\int_\Sigma |\sigma|^n = \theta |\mathbb{S}^n|$. Taking the product of $M$ with $\mathbb{R}^{m-1}$ then yields a codimension-$m$ submanifold with the same integral $\int_\Sigma |\sigma|^n = \theta |\mathbb{S}^n|$.

In the case where $n = 1$, we cannot even expect $\Sigma$ to have constant $|\sigma|$, since any closed curve in $\mathbb{R}^{m+1}$ that lies on a two-dimensional plane and is convex in that plane satisfies $\int_\Sigma|\sigma|=2\pi$ when regarded as a closed curve in $\mathbb{R}^{m+1}$. Indeed, this condition precisely characterizes the equality case of the classical Fenchel inequality.

To analyze the equality case in Theorem \ref{thm general fenchel willmore'}, from now on we are going to assume that the equality $\int_\Sigma|\sigma|^n =\theta|\mathbb S^n|$ holds.

\begin{lemma}\label{lem no focal}
For all $x \in \Sigma^+$, $z \in {T}_x^{\perp} \Sigma$ such that $|z|=1$ and $\langle z, \sigma(x)\rangle \le 0$,
the geodesic $\gamma(s)=\exp_x\left(sz\right)$ minimizes the distance from $\Sigma$ for all $s \ge0$. In particular, no focal point of $\Sigma$ occurs along $\gamma(s)$ for $s>0$ and the normal exponential map when restricted to $\{(x, z)\in T^\perp \Sigma^+: \langle z, \sigma(x)\rangle \le 0\}$ is injective, and hence is a diffeomorphism onto its image.
\end{lemma}

\begin{proof}
Define $\mathcal A_r=\{(x, z)\in T^\perp\Sigma^+, d\left(q, \exp_x(r z)\right) \ge r|z| \text { for all } q \in \Sigma \}$.
Suppose the conclusion is not true, then there exist $r_{0}>0$, $x_{0} \in \Sigma^+$, $z_{0}\in T_{x_0}^\perp \Sigma$, with $|z_{0}|=1$ and $\langle z_0, \sigma(x_0)\rangle \le 0$ such that $(x_{0}, z_{0})\notin \mathcal A_{r_{0}}$.

We claim that
$(x_0, z_0)\notin \mathcal A_{r}$ for all $r\ge r_0$. This is because there exists $q_0\in \Sigma$ such that
$d(q_0, \exp_{x_0}(r_0z_0) <r_0$ and hence for $r\ge r_0$,
\begin{align*}
d(q_0, \exp_{x_0}(rz_0)
\le& d(q_0, \exp_{x_0}(r_0z_0) +(r-r_0) \\
<& r_0 +(r-r_0)
=r, \quad \text{i.e.}(x_0, z_0)\notin \mathcal A_r.
\end{align*}

It then follows that there is a neighborhood $V$ of $(x_0, z_0)$ in $T^\perp \Sigma^+$ such that for $(x, z)\in V$, $(x, z)\notin \mathcal A_r$ for $r\ge r_0$. Define $W=V\cap \{(x, z)\in T^\perp \Sigma^+: |z|<1\}\ne \emptyset$.

As in \eqref{ineq est I0 I+},
\begin{equation}\label{ineq equality case1}
\begin{aligned}
& |\{p \in M: \alpha r < d(x, p) < r \text{ for all } x \in \Sigma\}|\\
\le& \int_{\Sigma^+} \int_{-1}^{\frac{1}{r|\sigma|}} \int_{Y_{\alpha, x, t}} \left|\det D\Phi_{r}(x, y, t)\right|\, 1_{A_r}(x, y, t) dy\, dt\, d\operatorname{vol}_{\Sigma}+I^0(\alpha, r) \\
\le& \int_{\Sigma^+} \int_{-1}^{ \frac{1}{r|\sigma|}} \int_{Y_{\alpha, x, t}} r^m(1-r|\sigma| t)^n 1_{A_r}(x, y, t) dy\, dt\, d\operatorname{vol}_{\Sigma}+I^0(\alpha, r) \\
\le& \int_{\Sigma^+} \int_{-1}^{ \frac{1}{r|\sigma|}} \int_{Y_{\alpha, x, t}} r^{m}\bigl(1-r|\sigma| t\bigr)^{n} dy\, dt\, d\operatorname{vol}_{\Sigma}\\
& - \int_{\Sigma^+} \int_{-1}^{ \frac{1}{r|\sigma|}} \int_{Y_{\alpha, x, t}}1_{W}(x, y, t) r^{m}\bigl(1-r|\sigma| t\bigr)^{n} dy\, dt\, d\operatorname{vol}_{\Sigma}+I^0(\alpha, r)\\
=:& J^+(\alpha, r) -J^W(\alpha, r)+I^0(\alpha, r),
\end{aligned}
\end{equation}
where $I^0$ is defined in \eqref{ineq est I0 I+}. It is not hard to see that (see the next lemma for a similar argument)
\begin{align*}
& \liminf_{\alpha \rightarrow 1^{-}} \liminf_{r \rightarrow \infty} \frac{1}{1-\alpha} \cdot \frac{1}{r^{n+m}}J^W(\alpha, r)\\
\ge& \liminf_{\alpha \rightarrow 1^{-}} \liminf_{r \rightarrow \infty} \frac{1}{1-\alpha} \cdot \frac{1}{r^{n+m}}\int_{\Sigma^+} \int_{-1}^{0} \int_{Y_{\alpha, x, t}} 1_W(x, y, t) r^m(1-r|\sigma| t)^n d y d t d \operatorname{vol}_{\Sigma}>0.
\end{align*}
We can divide \eqref{ineq equality case1} by $r^{n+m}$ and $1-\alpha$ and analyze the limit as before. In view of the proof of Theorem \ref{thm general fenchel willmore'}, the above strict inequality then implies that the equality $\int_\Sigma|\sigma|^n= \theta|\mathbb S^n|$ does not hold, a contradiction.
\end{proof}

\begin{lemma}\label{lem xyt=1}
For every $r>0, x \in \Sigma^+, z \in {T}_x^{\perp} \Sigma$ satisfying $|z|=1$ and $\langle z, \sigma(x)\rangle \le0$, we have
$$
\left|\det D \Phi_r(x, z)\right| \ge r^m\left(1-r\langle z, \sigma(x) \rangle \right)^n.
$$
\end{lemma}

\begin{proof}

Assume on the contrary that there exists $x_{0}\in \Sigma^+$, $z_{0}\in {T}_{x}^{\perp}\Sigma$ that satisfy $|z_{0}|=1$ and $\langle z_0, \sigma(x_0)\rangle \le0$, such that
$$
\left|\det D \Phi_{r_{0}}(x_{0}, z_0)\right| < r_{0}^m\left(1-r_0\langle z_0, \sigma(x_0)\rangle \right)^n
$$
for some $r_{0}>0$. Since this is an open condition, we can without loss of generality assume that $y_0\ne0$.

By continuity, there exist $\varepsilon\in(0, 1)$ and a neighbourhood $V$ of $(x_{0}, z_0)$ in $T^{\perp}\Sigma$ such that
$$
\left|\det D\Phi_{r_{0}}(x, z)\right|
< (1-\varepsilon) r_{0}^{m} \bigl(1-r_{0}\langle z, \sigma(x)\rangle \bigr)^{n}
\quad\text{on } V.
$$
Lemma \ref{lem no focal} and Lemma \ref{lem 2.4''} then imply that for every $r>r_{0}$,
$$
\left|\det D\Phi_{r}(x, z)\right|
< (1-\varepsilon) r^{m} \bigl(1-r \langle z, \sigma(x)\rangle \bigr)^{n}
\quad\text{on } V\cap A_{r}.
$$

Now, assume that $m\ge 3$ and write $(x, z)=(x, y, t)$ as before.
For $\alpha\in(0, 1)$ and each $(x, t)$, define
$$
Y_{\alpha, x, t}
=\bigl\{ y\in\widetilde T_{x}^{\perp}\Sigma:
\alpha^{2}<|y|^{2}+t^{2}<1\bigr\},
$$
and write $Y_{\alpha}=Y_{\alpha, x, t}$.

Consequently, by applying Lemma \ref{lem 2.2''}, and following the reasoning in \eqref{ineq est I0 I+}, \eqref{ineq I+}, \eqref{ineq I+2}, \eqref{ineq I+3} and \eqref{ineq I+4}, we have
\begin{equation}\label{ineq est equality case}
\begin{aligned}
& |\{p \in M: \alpha r<d(x, p)<r \text { for all } x \in \Sigma\}| \\
\le& \int_{\Sigma^+} \int_{-1}^{\frac{1}{r|\sigma|}} \int_{Y_\alpha}
\left|\det D \Phi_r(x, y, t)\right| 1_{A_r}(x, y, t) d y d t d \mathrm{vol}_{\Sigma} \\
\le& \int_{\Sigma^+} \int_{-1}^{\frac{1}{r|\sigma|}} \int_{Y_\alpha}
\big(1-\varepsilon 1_{V}(x, y, t)\big) r^m\left(1-r|\sigma| t\right)^n d y d t d \mathrm{vol}_{\Sigma}+I^{0}(\alpha, r)\\
\le&
(1-\alpha)\frac{\Gamma\left(\frac{n+1}{2}\right) \Gamma\left(\frac{m-1}{2}\right)}{2 \Gamma\left(\frac{n+m}{2}\right)} r^{m+n}\int_{\Sigma^+}|\sigma|^n\\
& - \varepsilon \int_{\Sigma^+} \int_{-1}^{\frac{1}{r|\sigma|}} \int_{Y_\alpha}
1_{V}(x, y, t) r^m\left(1-r|\sigma| t\right)^n d y d t d \mathrm{vol}_{\Sigma}\\
& +O(1-\alpha)O\left(r^{m+n-1}\right) +O((1-\alpha)^2)O(r^{m+n})\\
=&:J(\alpha, r)-\varepsilon I(\alpha, r)+\mathrm{remainder}
\end{aligned}
\end{equation}
for all $r>r_{0}$ and $\alpha$ close to $1$.

We now claim that $\displaystyle \liminf_{\alpha \to 1^{-}} \liminf_{r \to \infty} \frac{1}{1-\alpha}\cdot\frac{1}{r^{n+m}} I(\alpha, r)$ is positive. This is straightforward but slightly tedious.

As $V$ is open, for $\alpha$ sufficiently close to 1 the set $V \cap Y_{\alpha, x_0, t_0}$ contains, in polar coordinates on $\widetilde{T}_{x_0}^{\perp} \Sigma$, an open set
$\{(\rho, \theta):\ \alpha^{2}-t_{0}^{2}<\rho^{2}<1-t_{0}^{2}, \ \theta\in\mathcal O\}$,
where $\mathcal O\subset\mathbb S^{m-2}$ is an open neighbourhood of $\frac{y_{0}}{|y_{0}|}$ independent of $\alpha$.

Hence, for $\alpha$ close to $1$,
$$
|V\cap Y_{\alpha, x_{0}, t_{0}}|
\;\ge\;
\frac{|\mathcal O|}{m-1}\Bigl[(1-t_{0}^{2})^{\frac{m-1}{2}}-(\alpha^{2}-t_{0}^{2})^{\frac{m-1}{2}}\Bigr]
\ge2 \delta_1(1-\alpha)
$$
for some $\delta_{1}>0$ independent of $\alpha$. By shrinking $V$ if necessary, we may assume $|V\cap Y_{\alpha, x, t}|\ge\delta_{1}(1-\alpha)$ for all $(x, y, t)\in V$ and $\alpha$ close to $1$.

Consequently, for $\alpha$ near $1$,
\begin{equation}\label{ineq lower bound 2}
\frac{1}{1-\alpha}\cdot \frac{1}{r^{n+m}} I(\alpha, r) \ge \frac{\delta_{1}}{r^{n}}\int_{B_{\rho}(x_{0})} \int_{T_{x, r}\cap V} \bigl(1-r|\sigma| t\bigr)^{n}dt\, d\operatorname{vol}_{\Sigma},
\end{equation}
for some $\rho>0$, where $T_{x, r} =\{t\in \mathbb R: -1<t<\tfrac{1}{r|\sigma(x)|}\}$.

Because the integrand of the integral on the RHS in \eqref{ineq lower bound 2} is
decreasing in $t$, we have
\begin{equation}\label{ineq TV}
\begin{aligned}
\int_{T_{x_{0}, r} \cap V}
\bigl(1-r|\sigma(x_{0})| t\bigr)^{n} dt
\ge& \int_{ t_0 -\delta_{2}}^{ t_0 } \bigl(1-r|\sigma(x_{0})| t\bigr)^{n} dt \\
\ge& \int_{ -\delta_{2}}^{0} \bigl(1-r|\sigma(x_{0})| t\bigr)^{n} dt \\
=& \frac{1}{r|\sigma(x_{0})|(n+1)}
\left[\bigl(1+r|\sigma(x_{0})|\delta_{2}\bigr)^{n+1}-1\right]\\
\ge& \frac{\delta_2^{n+1}r^n|\sigma(x_0)|^{n}}{n+1},
\end{aligned}
\end{equation}
where $0<\delta_{2}<1$ is chosen so that $V \cap \{x=x_{0}, y=y_{0}\}$ contains the interval
$(t_{0}-\delta_{2}, t_{0})$.
By continuity,
$$
\int_{T_{x, r} \cap V}
\bigl(1-r|\sigma| t\bigr)^{n} dt
\ge\frac{\delta_{2}^{\, n+1} r^{n}}{2(n+1)} |\sigma(x_{0})|^{n}
$$
for $x$ near $x_{0}$.

Combining this with \eqref{ineq lower bound 2}, we find that
$$
\liminf_{\alpha \rightarrow 1^{-}} \liminf_{r \rightarrow \infty}
\frac{1}{1-\alpha}\cdot \frac{1}{r^{n+m}} I(\alpha, r) > 0.
$$

Putting this into \eqref{ineq est equality case} and proceeding as in the proof of Theorem \ref{thm general fenchel willmore'}, we conclude that
$$
\theta|\mathbb S^{n}|< \int_{\Sigma} |\sigma|^{n},
$$
which is a contradiction.

If $m=2$, note that the condition in Theorem \ref{thm general fenchel willmore'} is that the sectional curvature of $M$ is non-negative. By taking the product of $M$ with $\mathbb{R}$, we can view $\Sigma$ as a codimension $3$ submanifold. Observe that the product manifold $M \times \mathbb{R}$ still has nonnegative sectional curvature, while both $\theta$ and $\int_{\Sigma}|\sigma|^n$ remain invariant. The preceding argument can still be applied, which yields a contradiction.

When $m = 1$, the previous argument requires a slight modification. Observe that $z_0 = -\frac{\sigma(x_0)}{|\sigma(x_0)|}$, and hence, under our identification, $t_0 = -1$. We choose a sufficiently small neighborhood $V$ of $(x_0, t_0)$ as before, with the additional property that for every $(x, t)\in V$, and for all $\alpha$ sufficiently close to $1$ and $r$ sufficiently large, one has $T_{\alpha, r, x}\cap V_x \supset (-1, -\alpha)$. Here $T_{\alpha, r, x}:= \{ t \in \mathbb{R}: -1 < t < \tfrac{1}{r|\sigma(x)|}, \ \alpha < |t| < 1 \}$ and $V_x:= \{ t \in \mathbb{R}: (x, t)\in V \}$.

We replace \eqref{ineq est equality case} with
\begin{equation}\label{ineq est equality case'}
\begin{aligned}
& | \{p \in M: \alpha r<d(x, p)<r \text { for all } x \in \Sigma\} | \\
\le & \int_{\Sigma^{+}} \int_{T_{\alpha, r, x}} \left|\det D \Phi_r(x, t)\right| 1_{A_r}(x, t) d t d \operatorname{vol}_{\Sigma} \\
\le & \int_{\Sigma^{+}} \int_{T_{\alpha, r, x}} \left(1-\varepsilon 1_V(x, t)\right) r(1-r|\sigma| t)^n d t d \operatorname{vol}_{\Sigma}+I^0(\alpha, r)\\
=& \int_{\Sigma^{+}} \int_{T_{\alpha, r, x}} r(1-r|\sigma| t)^n d t d \operatorname{vol}_{\Sigma}
-\varepsilon\int_{\Sigma^{+}} \int_{T_{\alpha, r, x}\cap V_x} r(1-r|\sigma| t)^n d t d \operatorname{vol}_{\Sigma}
+I^0(\alpha, r).
\end{aligned}
\end{equation}
We replace \eqref{ineq TV} with
\begin{equation*}
\begin{aligned}
\int_{T_{\alpha, r, x_0 }\cap V_{x_0}}r \left(1-r\left|\sigma\left(x_0\right)\right| t\right)^n d t
\ge& r \int_{-1}^{-\alpha}\left(1-r\left|\sigma\left(x_0\right)\right| t\right)^n d t\\
=& \frac{\left(1+r\left|\sigma\left(x_0\right)\right|\right)^{n+1}-\left(1+\alpha r\left|\sigma\left(x_0\right)\right|\right)^{n+1}}{(n+1)\left|\sigma\left(x_0\right)\right|}\\
=& {r\left(1+r\left|\sigma\left(x_0\right)\right|\right)^n}(1-\alpha)+O(r^{n+1}(1-\alpha)^2)
\end{aligned}
\end{equation*}
if $\alpha$ is close enough to $1$ and $r$ sufficiently large.

By continuity, it is then easy to see that
$$
\liminf_{\alpha \rightarrow 1^{-}} \liminf_{r \rightarrow \infty} \frac{1}{1-\alpha} \cdot \frac{1}{r^{n+1}}\int_{\Sigma^{+}} \int_{T_{\alpha, r, x}\cap V_x} r(1-r|\sigma| t)^n d t d \operatorname{vol}_{\Sigma}>0.
$$
In view of \eqref{ineq est equality case'}, the same reasoning as in the previous argument applies, which leads to a contradiction.
\end{proof}

We are now ready to prove Theorem \ref{thm equality case}, the equality case of Theorem \ref{thm general fenchel willmore'}.
\begin{proof}

\noindent(a) and (b) \\
Suppose the equality holds. Then $\Sigma$ must be connected, for otherwise we can apply the argument to any of its component and get a bigger value of $\int_\Sigma|\sigma|^n$.

Fix $(x, z)=(x, y, t)$ such that $x\in \Sigma^+$, $t\le 0$ and $|y|^2+t^2=1$.
Define $\mathrm{A}=-\langle\mathrm{II}(x), y\rangle-t\left\langle\mathrm{II}(x), \frac{\sigma}{|\sigma|}\right\rangle$. There exists small enough $s_0>0$, such that $g_{\Sigma}+s \mathrm{A}>0$ for all $0<s<s_0$. We may then define the vectors $\left\{e_1, \cdots, e_{n+m}\right\}$ and hence the family of matrices $P(s)$ along ${\gamma}(s)=\exp_{ x }\left(s\left(y+t\frac{\sigma}{|\sigma|}\right)\right)$, in a manner analogous to the construction in the proof of Lemma \ref{lem 2.4''}.
For small enough $s$, we have $\left|\det D \Phi_s(x, y, t)\right|= |\det P(s)| \ge s^m\left(1-s t|\sigma(x)|\right)^n$ by Lemma \ref{lem xyt=1}. Hence
$$
\det P(s) \ge s^m\left(1-s t|\sigma(x)|\right)^n>0
$$
for $s \in\left(0, s_0\right)$, by making $s_0$ smaller if necessary.

This implies $P(s)>0$ on $(0, s_0)$. We can then define $Q(s)= P^{\prime}(s)P(s)^{-1}$ for $s \in\left(0, s_0\right)$ and follow the analysis of the Riccati equation as in the proof of Lemma \ref{lem 2.4''} to deduce that for $s\in(0, s_0)$,
\begin{equation}\label{ineq det P}
\det P(s) \le s^m(1-s t|\sigma(x)|)^n.
\end{equation}
Therefore,
\begin{equation*}
\begin{aligned}
\det P(s) =s^m(1-s t|\sigma(x)|)^n
\end{aligned}
\end{equation*}
on $(0, s_0)$.

Using the same notation and Riccati inequality analysis as in the proof of Lemma~\ref{lem 2.4''} leading to \eqref{ineq det P}, and applying Lemma~\ref{lem riccati} together with \eqref{ineq cs}, we conclude that the matrix $Q(s)=P^{\prime}(s) P^{-1}(s)$ satisfies
\begin{equation}\label{eq Q}
Q(s)=\begin{bmatrix}
-\frac{\langle\sigma(x), z\rangle}{1-s\langle\sigma(x), z\rangle}I_n& 0\\
0& \frac{1}{s}I_m
\end{bmatrix}
\end{equation}
and
\begin{equation}\label{eq R}
\sum_{i=1}^n\left\langle\bar{R}\left(E_i(s), {\gamma}^{\prime}(s)\right) {\gamma}^{\prime}(s), E_j(s)\right\rangle=0, \quad
\sum_{\alpha=n+1}^{n+m}\left\langle\bar{R}\left(E_\alpha(s), {\gamma}^{\prime}(s)\right) {\gamma}^{\prime}(s), E_\alpha(s)\right\rangle=0
\end{equation}
for $s\in(0, s_0)$.

By \eqref{Q0}, it then follows that $\left\langle\mathrm{II}, z\right\rangle=\langle \sigma, z\rangle g_\Sigma$.
Since $x \in \Sigma^{+}$ and $z=y+t\frac{\sigma}{|\sigma|}\in T^\perp_x\Sigma$ with $t\le0$ and $|y|^2+t^2=1$ are arbitrary, we conclude that
for $x\in \Sigma^+$, $\mathrm{II}(x)=\sigma(x) g_\Sigma$. i.e. $\Sigma^+$ is umbilical.

By taking the trace of the Codazzi equation $\nabla_X \mathrm{II}(Y, Z)-\nabla_Y \mathrm{II}(X, Z)=(\bar{R}(X, Y) Z)^\perp$ on $\Sigma^+$, and using $\mathrm{II}(x)=\sigma(x) g_{\Sigma}$, we have
\begin{align*}
(n-1)\nabla^\perp_{e_i}\sigma=\sum_{j=1}^n \left(\bar{R}\left(e_i, e_j\right) e_j\right)^{\perp},
\end{align*}
where $\{e_i\}_{i=1}^{n}$ is a local orthonormal frame on $\Sigma^+$ and $\nabla^\perp$ is the normal connection.

Fix a unit vector $z\in T_x^\perp \Sigma^+$ and define the symmetric bilinear form $\mathrm{S}$ on $T_x\Sigma \oplus\operatorname{span} (z)$ by $\mathrm{S}(U, V)=\sum_{i=1}^{n}\langle \bar R(U, e_i)e_i, V\rangle+\langle \bar R(U, z)z, V\rangle$. Then it is easy to see that the assumption $\overline {\mathrm{Ric}}_k\ge0$ implies $\overline {\mathrm{Ric}}_{n}\ge 0$, which in turn implies that $\mathrm{S}$ is non-negative.
By taking $s\to0^+$ in the first equation of \eqref{eq R}, we have
$\mathrm{S}(z, z)=0$, and so $\mathrm{S}(z, \cdot)=0$ on $T_x\Sigma\oplus\operatorname{span}(z)$. We then deduce that
\begin{equation*}
(n-1) \langle \nabla_{e_i}^{\perp} \sigma, z\rangle =\sum_{j=1}^n \langle \bar{R}\left(e_i, e_j\right) e_j, z\rangle=0.
\end{equation*}
Since the unit vector $z\in T_x^\perp \Sigma$ is arbitrary, we deduce that the normalized mean curvature vector $\sigma$ is parallel if $n>1$.

Let $n>1$. $\Sigma^+$ is obviously an open subset in $\Sigma$. Take $\mathring{\Sigma}$ to be any component of $\Sigma^+$. Let $p_0$ be a limit point of $\mathring{\Sigma}$, then we can take a sequence in $\mathring{\Sigma}$ converging to $p_0$. Along this sequence $|\sigma|=c>0$ is constant, and hence $|\sigma(p_0)|=c$ and $p_0$ lies in the open set $\Sigma^+$. Therefore $p_0\in\mathring{\Sigma}$. Consequently, $\mathring{\Sigma}$ is a non-empty open and closed subset, and hence $\Sigma^+=\Sigma$. Therefore, in the rest of the proof below, we can replace $\Sigma^+$ by $\Sigma$ if $n>1$.

We know from Lemma \ref{lem no focal} that the normal exponential map $\exp$ is a diffeomorphism from $\{(x, z)\in T^\perp \Sigma^+: \langle z, \sigma(x)\rangle \le 0\}$ onto its image. In particular, $\Sigma^+$ is an embedded submanifold.

Using the notation in the proof of Lemma \ref{lem 2.4''}, the equation $P'(s)=Q(s)P(s)$ together with \eqref{eq Q} then implies that the Jacobi fields $X_i(s)$ along ${\gamma}$ defined by
\begin{equation*}
\begin{aligned}
X_i(0)=& e_i \text { and } X_i^{\prime}(0)=\lambda_i e_i \text{ for } i=1, \cdots, n, \\
X_\alpha(0)=& 0 \text { and } X_\alpha'(0)=e_\alpha \text{ for }\alpha=n+1, \cdots, n+m
\end{aligned}
\end{equation*}
satisfy
\begin{align*}
X_i(s)= \left(1-s\langle\sigma(x), z\rangle\right) E_i(s) \text{ for } i=1, \cdots, n
\end{align*}
and
\begin{align*}
X_\alpha(s)= s E_\alpha(s)\text{ for }\alpha=n+1, \cdots, n+m.
\end{align*}

We conclude that the pullback metric of $M$ by the normal exponential map to the subset
$$\left\{(x, z) \in T^{\perp} \Sigma^+:\langle z(x), \sigma(x)\rangle \le 0\right\}=\{(x, y, t)\in \widetilde T^\perp \Sigma^+\times \mathbb R: t\le0\}$$ of the normal bundle of $\Sigma$ is given by
\begin{equation}\label{eq g}
g(x, y, t)=d t^2+d y^2+(1-t|\sigma(x)|)^2 g_{\Sigma}.
\end{equation}

When $n>1$, then $|\sigma|=\left(\frac{\theta\left|\mathbb{S}^n\right|}{|\Sigma|}\right)^{\frac{1}{n}}>0$ is constant and by the change of variable $r=r_0-t$, the pullback metric can also be expressed in the form
\begin{align*}
g=dr^2+dy^2+\left(\frac{r}{r_0}\right)^2g_\Sigma,
\end{align*}
where $r_0=\left(\frac{|\Sigma|}{\theta|\mathbb S^n|}\right)^{\frac{1}{n}}$.

Conversely, it is easy to see that if $\Sigma$ and $g$ are given above, then $\int_\Sigma|\sigma|^n=\theta |\mathbb S^n|$.

\noindent(c)
We now show that the image of the normal exponential map applied to normal vectors in the half-space opposite the mean curvature vector fills the tubular neighbourhood of $\Sigma$ with asymptotic volume density one. A precise formulation is as follows.

Let $\Sigma_r^+$ denote the subset of the normal bundle of $\Sigma^+$ given by
$$\Sigma_r^+:= \{(x, z) \in T^\perp \Sigma^+: \langle \sigma(x), z \rangle \le 0, \; |z| < r\}. $$
We compute the volume of the image of $\Sigma_r^+$ under the normal exponential map. By \eqref{eq g},
\begin{align*}
|\exp(\Sigma^+_r)|
=& \int_{\Sigma_r^+}\exp^* \mathrm{vol}_g\\
=& \int_{\Sigma^+}\int_{-1}^0 \int_{Y_{x, t}}\left|\det D \Phi_r(x, y, t)\right| d y d t d \operatorname{vol}_{\Sigma}(x)\\
=& \int_{\Sigma^{+}} \int_{-1}^0 \int_{Y_{x, t}} r^m(1-r|\sigma(x)| t)^n d y d t d \operatorname{vol}_{\Sigma}(x) \\
=& \left|\mathbb{B}^{m-1}\right| r^m \int_{\Sigma^{+}} \int_{-1}^0\left(1-t^2\right)^{\frac{m-1}{2}}(1-r|\sigma(x)| t)^n d t d \operatorname{vol}_{\Sigma}(x).
\end{align*}
Here $Y_{x, t}:=\left\{y \in \widetilde{T}_x^{\perp} \Sigma:|y|^2+t^2<1\right\}$.

Consider the fiber integral
\begin{equation*}
\begin{aligned}
\int_{-1}^0\left(1-t^2\right)^{\frac{m-1}{2}}(1-r|\sigma(x)| t)^n d t & =r^n|\sigma(x)|^n \int_0^1 u^n\left(1-u^2\right)^{\frac{m-1}{2}} d u+O\left(r^{n-1}\right) \\
& =\frac{\Gamma\left(\frac{n+1}{2}\right) \Gamma\left(\frac{m+1}{2}\right)}{2 \Gamma\left(\frac{n+m+2}{2}\right)}|\sigma(x)|^n r^n+O\left(r^{n-1}\right)
\end{aligned}
\end{equation*}
as $r\to \infty$.

Therefore, using $\int_{\Sigma^+}|\sigma|^n=\theta|\mathbb S^n|$, we arrive at
\begin{align*}
|\exp(\Sigma^+_r)|
=& \frac{\Gamma\left(\frac{n+1}{2}\right) \Gamma\left(\frac{m+1}{2}\right)}{2 \Gamma\left(\frac{n+m+2}{2}\right)}|\mathbb B^{m-1}|r^{n+m}\int_{\Sigma^+} |\sigma(x)|^n +O(r^{n+m-1})\\
=& \theta \frac{\Gamma\left(\frac{n+1}{2}\right) \Gamma\left(\frac{m+1}{2}\right)}{2 \Gamma\left(\frac{n+m+2}{2}\right)}|\mathbb B^{m-1}| |\mathbb S^n|r^{n+m}+O(r^{n+m-1}).
\end{align*}
On the other hand, we have the volume estimate for the tubular neighborhood of radius $r$ around $\Sigma$, denoted by $N_r(\Sigma) =
\{p \in M: d(p, \Sigma) < r \}$:
\begin{align*}
|N_r(\Sigma)|=|\{p \in M: d(p, \Sigma)< r \}| = \theta |\mathbb B^{n+m}|r^{n+m}(1+o(1))
\end{align*}
as $r\to \infty$, where $o(1)\to 0$ as $r\to \infty$.

Obviously, $\exp(\Sigma^+_r)\subset N_r(\Sigma)$.
The volume ratio satisfies
\begin{align*}
\frac{|\exp(\Sigma^+_r)|}{|N_r(\Sigma)|}
=& \frac{\Gamma\left(\frac{n+1}{2}\right) \Gamma\left(\frac{m+1}{2}\right)}{2 \Gamma\left(\frac{n+m+2}{2}\right)}
\cdot\frac{\left|\mathbb{B}^{m-1}\right|\left|\mathbb{S}^n\right|}{\left|\mathbb{B}^{n+m}\right|} +o(1).
\end{align*}
Using $\left|\mathbb{B}^k\right|=\frac{\pi^{\frac{k}{2}}}{\Gamma\left(\frac{k}{2}+1\right)}$ and $\left|\mathbb{S}^n\right|=\frac{2 \pi^{\frac{ n+1 }{2}}}{\Gamma\left(\frac{n+1}{2}\right)}$, we compute
\begin{align*}
\frac{\Gamma\left(\frac{n+1}{2}\right) \Gamma\left(\frac{m+1}{2}\right)}{2 \Gamma\left(\frac{n+m+2}{2}\right)} \cdot \frac{\left|\mathbb{B}^{m-1}\right|\left|\mathbb{S}^n\right|}{\left|\mathbb{B}^{n+m}\right|}
=& \frac{\Gamma\left(\frac{n+1}{2}\right) \Gamma\left(\frac{m+1}{2}\right)}{2 \Gamma\left(\frac{n+m+2}{2}\right)} \cdot \frac{2 \pi^{\frac{ m-1 }{2} }\pi^{\frac{ n+1 }{2}}}{\Gamma\left(\frac{m+1}{2}\right) \Gamma\left(\frac{n+1}{2}\right)} \cdot \frac{\Gamma\left(\frac{n+m+2}{2}\right)}{\pi^{ \frac{n+m}{2} }}\\
=& 1.
\end{align*}
We conclude that
$$
\lim_{r\to\infty}\frac{\left|\exp \left(\Sigma_r^{+}\right)\right|}{\left|N_r(\Sigma)\right|}=1.
$$
i.e. the image of $\left\{(x, z) \in T^{\perp} \Sigma^{+}:\langle z, \sigma(x)\rangle \le 0\right\}$ under the exponential map has asymptotic volume density one in $M$.
\end{proof}

\end{document}